\documentclass[11pt,a4paper]{article}

\usepackage[T1]{fontenc}
\usepackage[utf8]{inputenc}
\usepackage{csquotes}

\usepackage{amsfonts, amsmath, amssymb, amsthm, dsfont, mathrsfs, latexsym, euscript}

\usepackage[numbers]{natbib}

\usepackage{tikz}
\usepackage{xcolor}
\usepackage{caption, subcaption}
\usepackage[colorlinks=true, citecolor=blue, linkcolor=black, urlcolor=blue]{hyperref}
\usepackage[left=3cm,right=3cm,top=2cm,bottom=2cm]{geometry}

\def\EE{\mathbb{E}}
\def\FF{\mathbb{F}}

\def\LL{\mathbb{L}}

\def\NN{\mathbb{N}}

\def\PP{\mathbb{P}}

\def\RR{\mathbb{R}}

\def\Bcal{\mathcal{B}}

\def\Fcal{\mathcal{F}}

\def\Ical{\mathcal{I}}

\def\Ncal{\mathcal{N}}

\def\Pcal{\mathcal{P}}

\def\Scal{\mathcal{S}}
\def\Tcal{\mathcal{T}}

\def\Wcal{\mathcal{W}}

\newcommand{\eps}{\varepsilon}
\def\E#1{\mathbb{E}\left[ #1 \right]}

\def\1{\mathds{1}}

\newtheorem{theorem}{Theorem}[section]
\makeatletter
\renewcommand{\thetheorem}{%
  \ifnum\c@subsection=0
    \thesection.\number\c@theorem
  \else
    \thesubsection.\number\c@theorem
  \fi}
\makeatother

\newtheorem{lemma}[theorem]{Lemma}

\newtheorem{definition}[theorem]{Definition}

\newtheorem{remark}[theorem]{Remark}

\numberwithin{equation}{section}
\makeatletter
\renewcommand{\theequation}{%
  \ifnum\c@subsection=0
    \thesection.\number\c@equation
  \else
    \thesubsection.\number\c@equation
  \fi}
\makeatother

\def\z{\mathrsfs{Z}}

\DeclareSymbolFontAlphabet{\mathrsfs}{rsfs}

\DeclareMathOperator*{\esssup}{ess\,sup}

\title{Gaussian approximation for the integrated intensity of counting processes: application to Hawkes processes}
\author{Benjamin Massat\footnote{Sorbonne Université, Laboratoire de Probabilités, Statistique et Modélisation, LPSM, UMR 8001, Paris, France. \; Email: \texttt{massat@lpsm.paris.fr}}{~}\footnote{Universit\'e de Toulouse, Insitut de Mathématiques de Toulouse, IMT, UMR 5219, Toulouse, France. \texttt{benjamin.massat@math.univ-toulouse.fr}\\  \url{https://benjamin-massat-439c85.pages.math.cnrs.fr/}} }

\date{\today}

\begin{document}
\maketitle

\begin{abstract}
In this paper, we establish quantitative central limit theorems for the integrated fluctuations of the intensity of a counting process. Building on the work of Coutin and using the Clark–Ocone formula, we first derive a general upper bound on the Wasserstein distance between a Gaussian random variable and the normalized integral of the centered intensity. We then apply this result to subcritical linear Hawkes processes. In this setting, we prove that the integrated fluctuations of the intensity converge in distribution to a centered Gaussian random variable, with a convergence rate governed by the tail behavior of the resolvent associated with the underlying Volterra equation.
\end{abstract}


\section{Introduction}Introduced by \cite{hawkes_spectra_1971}, the Hawkes process $H$ is a self-exciting point process whose intensity, denoted by $\lambda$, is given by
\begin{equation}
\lambda_t = \mu + \int_{(0,t)} \phi(t-s) dH_s, \quad t > 0,
\end{equation}
where $\mu>0$ and $\phi:\RR_+\to\RR_+$ is a kernel function. While the linear Hawkes process was the first to be studied, nonlinear extensions have also been extensively investigated (see, for instance, \cite{bremaud_stability_1996,zhu_central_2013}). Over the past two decades, a large literature has been devoted to the limiting behaviour of Hawkes processes and, more recently, to quantitative versions of their limit theorems.

\subsection{Limit theorems for Hawkes processes}

In the linear case, \cite{bacry_limit_2013} established a law of large numbers and a functional central limit theorem (fCLT) for multivariate Hawkes processes observed over $[0,T]$ as $T\to+\infty$. In particular,
\begin{equation}
\label{eq:introcvs}
\left(
\frac{H_{tT}- \int_0^{tT} \lambda_s ds}{\sqrt{T}}
\right)_{t\in[0,1]}
\xrightarrow[]{T\to+\infty}
\left(\sigma B_t\right)_{t\in[0,1]},
\end{equation}
where the convergence holds in distribution with respect to the Skorokhod topology, $B$ is a standard Brownian motion, and $\sigma>0$. This result was subsequently extended to nonlinear Hawkes processes by \cite{zhu_nonlinear_2013}. In the linear case, the limiting variance is explicit and, in the univariate case, $\sigma^2=\frac{\mu}{(1-\|\phi\|_1)}$,
whereas the corresponding quantity is generally not explicit in the nonlinear setting. Similar limit theorems have subsequently been obtained in more general settings. Indeed, \cite{horst_functional_2024} established functional and scaling limit theorems under minimal assumptions on the kernel, including both subcritical ($ \left\| \phi \right\|_{\LL^1} < 1$) and critical regimes ($ \left\| \phi \right\|_{\LL^1} = 1$), while \cite{xu_scaling_2024} studied scaling limits for multivariate Hawkes processes.\\

The first quantitative result for one-dimensional marginals of nonlinear Hawkes processes was obtained by \cite{torrisi_gaussian_2016}. However, the resulting bound was not sufficient to quantify the convergence in \eqref{eq:introcvs}, since the corresponding Wasserstein bound does not vanish as $T\to+\infty$.

A series of works by \cite{hillairet_malliavin-stein_2022}, \cite{khabou_malliavin-stein_2021}, and \cite{privault_stein_2018} subsequently established quantitative central limit theorems for one-dimensional marginals of linear Hawkes processes. These results concern the compensated counting process, rather than the integrated fluctuations of the intensity considered in the present paper. More precisely, they obtained a bound of the form
\begin{equation}
\label{eq: maj_intro}
\Wcal_1\left(
\frac{H_T-\int_0^T\lambda_tdt}{\sqrt{T}},
\mathcal{N}(0,\sigma^2)
\right)
\leq \frac{C}{\sqrt{T}}.
\end{equation}
We also note that the works of \cite{hillairet_malliavin-stein_2022} and \cite{coutin_normal_2026} provide quantitative $L^1$-bounds for the time-averaged intensity (in the linear and nonlinear settings). More precisely, they establish bounds of the form
\begin{equation}
    \E{\left| \frac{1}{T}\int_0^T \lambda_t -\E{\lambda_t} dt \right|} \leq \frac{C}{\sqrt{T}}.
\end{equation}
These results quantify the law of large numbers for the integrated intensity. However, they do not describe its fluctuations at the central limit scale.
Their approach to prove \eqref{eq: maj_intro} relies on the Poisson imbedding representation of Hawkes processes together with a Malliavin's calculus adapted to the Poisson space, combined with Stein's method. We refer to \cite{bremaud_stability_1996} for the Poisson imbedding of point processes and to \cite{peccati_steins_2010} for Malliavin calculus on the Poisson space.

This Malliavin-Stein approach was further developed by \cite{coutin_normal_2026}, where new Wasserstein bounds were obtained for functionals of point processes. Their results apply, in particular, to nonlinear Hawkes processes and discrete linear Hawkes processes (see \cite{kirchner_hawkes_2016,quayle_etude_2022}). They also relaxed some of the assumptions required in the functional central limit theorem of \cite{zhu_nonlinear_2013}.

More recently, \cite{massat_limit_2026} obtained quantitative functional central limit theorem for nonlinear Hawkes processes in more general settings, including marked and compound Hawkes processes, using the Malliavin-Stein method. These results provide quantitative bounds at the functional level and thus go beyond the one-dimensional marginal estimates considered in the aforementioned works. \cite{massat_limit_2026} also relaxed the assumptions of \cite{zhu_nonlinear_2013}, thereby extending the corresponding functional central limit theorem for nonlinear Hawkes processes. 

Limit theorems for nearly unstable Hawkes processes have also been
established in \cite{jaisson_limit_2015,coutin_quantification_2025,liu_scaling_2025,horst_functional_2024}, with quantitative versions subsequently obtained in \cite{coutin_quantification_2025,horst_functional_2024}.

\subsection{Main contributions}
The main contributions of this paper can be summarized as follows.

\begin{itemize}
\item Based on the work of \cite{coutin_normal_2026} and by making use of Clark-Ocone formula, we establish a general upper bound between a Gaussian random variable and the integrated fluctuations of the intensity of a counting process:
$$\frac{1}{\sqrt{T}}\int_0^T
        \left(\lambda_t-\E{\lambda_t}\right)dt.$$
\item We apply this general result to the intensity of a subcritical linear Hawkes process. As a consequence, we prove a central limit theorem for the integrated fluctuations of the intensity showing that there exists $C>0$ and $\sigma^2>0$ such that for any $T>0$,
$$\Wcal_1\left(\frac{1}{\sqrt{T}}\int_0^T
        \lambda_t-\E{\lambda_t}dt, \, \Ncal(0,\sigma^2)\right) \leq \left(\frac{1}{\sqrt{T}}+ \int_0^T \psi(t) dt \right),$$
where $\psi$ is the resolvent associated with the underlying Volterra equation (see Theorem \ref{thm:appli_hawkes}).
To the best of our knowledge, this is the first central limit theorem for this quantity in the setting of linear Hawkes processes.

\end{itemize}

\section{Notations and preliminaries}
\label{sec: Notation}
To start with, let us introduce some notations. We denote by $\NN$ (resp. $\NN^*$) the set of non-negative (resp. positive) integers, that is $\NN := \{0, \, 1,\, 2,\dots \}$ (resp. $\NN^* := \{ 1,\, 2,\dots$ \}). We also define $\RR_+$ and $\RR_+^*$ as, respectively, $\RR_+ := [0, +\infty)$ and $\RR_+^* := (0, +\infty)$ the set of non-negative and positive real numbers. Moreover, for any $(a,b)\in \RR^2$, we may use the notation $[a,b]$ either for $a\leq b$ or for $a\geq b$. We also denote by $\Bcal(E)$ the Borel set of some topological space $E$.\\
Regarding function spaces, for any $A \in \Bcal(\RR)$ and $p \in \NN^*$, we define the Lebesgue space  
$$ \LL^p(A) = \left\{ f:A \to \RR \mid \|f\|_{\LL^p(A)} := \left(\int_A |f(t)|^p dt \right)^{1/p} < +\infty \right\}.$$  
For $p = \infty$, we set the set of a.e. bounded functions
$$ \LL^{\infty}(A) = \left\{ f:A \to \RR \mid \|f\|_{\LL^{\infty}(A)} := \esssup_{t\in A} |f(t)| < +\infty \right\}. $$  
In the specific case $A = \RR_+$, we use the shorthand notation $\LL^p := \LL^p(\RR_+)$ for any $p \in \NN^* \cup \{\infty\}$.

Throughout this paper, $C$ denotes a constant that may change from line to line and which is independent of $T$.\\

\subsection{Elements of stochastic analysis on the Poisson space}

In this section, we begin by recalling the basic ingredients of Malliavin's calculus that will serve as the foundation of our analysis. We first illustrate the main idea in the simple setting of a Poisson random variable. Indeed, Malliavin's calculus is closely related to integration-by-parts formulas associated with Gaussian and Poisson distributions. For $\Pcal$ a Poisson random variable with parameter $p>0$ and for any $f:\NN \to \RR$ such that $\Pcal f(\Pcal)$ and $f(\Pcal+1)$ are integrable, we have
$$\E{\Pcal f(\Pcal)} = p \E{f(\Pcal+1)}.$$
The term $f(\Pcal+1)$ reflects the effect of adding one point to the Poisson configuration. This motivates the introduction of the shift operator. We then introduce the Malliavin derivative 
$D$ in this context as 
$$D\left(f(\Pcal)\right) := f(\Pcal+1) - f(\Pcal).$$
The integration-by-part formula can then be written with the Malliavin derivative as
$$\E{\left(\Pcal-p\right) f(\Pcal)} = p \E{D\left(f(\Pcal)\right)}.$$
This formulation highlights the role of the centered Poisson random variable $(\Pcal-p) := \delta\left( \Pcal \right)$, which naturally appears in the integration-by-parts relation. We call it the divergence operator.\\\\
The same idea extends naturally to a Poisson random measure as proved by \cite{nualart_anticipative_1990}, \cite{peccati_steins_2010} or \cite{wu_construction_1987}.
In our context, the Poisson random measure will be described using two variables: $t$ for the jump times and an auxiliary parameter $\theta$, which plays a role in the imbedding representation used later.\\
We start by defining the space of configurations
$$ \Omega:=\left\{\omega=\sum_{i=1}^{n} \delta_{(t_{i},\theta_i)} \mid 0<t_1 < \cdots < t_n, \; (\theta_i)_{i=1,\dots,n} \in \left(\RR_+\right)^n, \; n\in \NN \cup\{+\infty\} \right\}.$$
Each realization of a counting process is represented as an element $\omega$ in $\Omega$ which is an $\NN$-valued measure on $\RR_+ \times \RR_+$. 

Let $\Fcal^N$ be the $\sigma$-field associated to the vague topology on $\Omega$, and $\PP$ be a probability measure under which the random measure $N$, defined as:
$$ N(B)(\omega):=\omega(B), \quad B \in \Bcal \left( \RR_+\times \RR_+\right),$$
is a random measure with intensity $1$ (so that $N(B)$ is a Poisson random variable with intensity $\pi(B)$ for any $B\in \Bcal \left( \RR_+ \times \RR_+ \right)$ and where $\pi$ denotes the Lebesgue measure). We set $\FF^N:=(\Fcal_t^N)_{t\in \RR_+}$ the natural history of $N$, that is 
$$\Fcal_t^N:=\sigma \left(N( \Tcal  \times \overline{\Tcal}), \; \Tcal \subset \Bcal((-\infty,t]), \; \overline{\Tcal} \in \Bcal(\RR_+) \right).$$
Let also, $\Fcal_\infty^N:=\lim_{t\to+\infty} \Fcal_t^N$. The expectation with respect to $\PP$ is denoted by $\E{\cdot}$. For $t\in \RR_+$, we denote by $\EE_t[\cdot]$ the conditional expectation $\E{\cdot \vert \Fcal_t^N}$.\\\\
\noindent
We then give some elements of stochastic analysis on the Poisson space, especially the shift operator, the Malliavin derivative and its dual operator: the divergence. 
 The following essentially comes from \cite{hillairet_malliavin-stein_2022} which provides some elements of Malliavin calculus for the Hawkes processes. Besides, note that their results do not change in our setting and that is why we decide not to provide the proofs here.

\begin{definition}[Shift operator]
\label{definition:shift}
    We define for $(t,\theta)$ in $\RR_+ \times \RR_+$ the measurable map $\eps_{(t,\theta)}^+ : \Omega \to \Omega $ where for any $A$ in $\Bcal(\RR_+ \times \RR_+)$
    $$(\eps_{(t,\theta)}^+(\omega))(A) := \left\lbrace \begin{array}{l} \omega(A \setminus \{(t,\theta)\}) +1, \quad \textrm{if } (t,\theta)\in A,\\\omega(A \setminus \{(t,\theta)\}) , \quad \textrm{else.}\end{array}\right. .$$
\end{definition}

\begin{lemma}
    \label{lemma:mesur}
    Let $t \in \RR_+$ and $F$ be an $\Fcal_t^N$-measurable random variable. Let $v > t$ and $\theta\geq 0$. It holds that 
    $$ F\circ\eps_{(v,\theta)}^+ = F, \quad \PP-a.s. $$
\end{lemma}

\begin{definition}[Malliavin derivative]
    For $F$ in $\LL^2(\Omega,\Fcal_\infty^N,\PP)$, we define $D F$ the Malliavin derivative of $F$ as 
    $$ D_{(t,\theta)} F := F\circ \eps_{(t,\theta)}^+ - F, \quad (t,\theta) \in \RR_+^2.  $$
\end{definition}

\noindent The following definition is a by-product of \cite[Theorem 1]{picard_formules_1996} (see also \cite{nualart_anticipative_1990}).

\begin{definition}
    Let $\Ical$  be the sub-sigma field of $\Bcal(\RR_+^2)\otimes \Fcal^N$ of stochastic processes $Z:=(Z_{(t,\theta)})_{(t,\theta) \in \RR_+ \times \RR_+}$ in $L^1(\Omega \times \RR_+ \times \RR_+,\PP\otimes \pi)$ such that 
    $$ D_{(t,\theta)} Z_{(t,\theta)} = 0, \quad \textrm{ for a.s. } (t,\theta) \in \RR_+ \times \RR_+.$$
\end{definition}

\begin{remark}
\label{rem:Nshift}
    Let $(t_0,\theta_0)$ in $\RR \times \RR_+$, $(s,t)$ in $\RR_+ \times \RR_+$ with $t_0 < s < t$.\\
    For $\mathcal T \in \{(s,t), (s,t], [s,t), [s,t]\}$ and $\overline{\Tcal}$ in $\Bcal(\RR_+)$,  we have: 
    $$ N\circ \eps_{(t_0,\theta_0)}^+ (\Tcal \times \overline{\Tcal}) = N (\Tcal \times \overline{\Tcal}). $$
\end{remark}

\begin{definition}
\label{def: div op}
    We set $\Scal$ the set of stochastic processes $Z:=(Z_{(t,\theta)})_{(t,\theta) \in \RR\times \RR_+}$ in $\Ical$ such that $Z$ is predictable according to the natural history and: 
    $$ \E{\iint_{\RR\times \RR_+} \left|Z_{(t,\theta)}\right|^2 dt d\theta} + \E{\left(\iint_{\RR\times \RR_+} Z_{(t,\theta)} dt d\theta \right)^2}<+\infty .$$
    For $Z$ in $\Scal$, we set the divergence operator with respect to $N$ as  
    \begin{equation}
    \label{eq:delta}
        \delta(Z):=\iint_{\RR\times \RR_+} Z_{(t,\theta)} N(dt,d\theta) - \iint_{\RR\times \RR_+} Z_{(t,\theta)} dt d\theta.
    \end{equation}
\end{definition}

\section{Main result}
\label{sec: gen estimate}

In this section, we provide a general upper bound of the $1$-Wasserstein distance between a Gaussian random variable and a random variable 
$$\frac{\int_0^T \lambda_s - \E{\lambda_s} ds}{\sqrt{T}},$$
where $\lambda$ stands for a positive $\FF^N$-predictable process.

\begin{theorem}
\label{thm: main}
    Let $\lambda$ be a positive $\FF^N$-predictable process such that $\sup_{x\in \RR_+} \E{\lambda_x} < +\infty$. In addition, assume that there exists a deterministic function $f: \RR_+ \to \RR_+$ satisfying 
    $$f\in \LL^1(\RR_+)$$
    and such that, for every $(u,s,\rho)\in \RR_+^3$ with $u<s$, 
    $$\1_{\rho \leq \lambda_u}\EE_{u^-} \left[ D_{(u,\rho)}(\lambda_s)\right] =\1_{\rho \leq \lambda_u} f(s-u).$$
    Then, for any $\sigma>0$, there exists $C>0$, such that for any $T>0$
    \begin{equation}
        \Wcal_1\left(\frac{\int_0^T \lambda_s - \E{\lambda_s} ds}{\sqrt{T}}, \Ncal(0,\sigma^2)\right) \leq \E{  \left|\sigma^2- \frac{1}{T}\int_{0}^T \lambda_u \left(\int_u^T f(s-u) ds \right)^2 du  \right|} + \frac{C}{\sqrt{T}}.
    \end{equation}
\end{theorem}

\begin{proof}
    For $T>0$, we set 
    \begin{equation}
        F^T := \frac{\int_0^T \lambda_s - \E{\lambda_s} ds}{\sqrt{T}}.
    \end{equation}
    By the Clark-Ocone formula (see for instance \cite{zhang_clarkocone_2009}), for any $s\in [0,T]$,
    $$\lambda_s - \E{\lambda_s} = \int_{\RR_+^2} \left[D_{(u,\rho)}\left( \lambda_s\right) \right]^\mathrsfs{P} \left( N(du,d\rho) - du d\rho \right) $$
    where $\left[D_{(u,\rho)}\left( \lambda_s\right) \right]^\mathrsfs{P}$ is the predictable projection of the Malliavin derivative of $\lambda_s$. Moreover, for fixed $(u,\rho)\in \RR_+^2$, this predictable projection is almost surely equal to the conditional expectation of the Malliavin derivative of $\lambda_s$ when it is conditioned by $\Fcal^N_{u^-}$ (see Remark 3.5 of \cite{zhang_clarkocone_2009}):
    $$\left[ D_{(u,\rho)}\left(\lambda_s \right)\right]^\mathrsfs{P} = \EE_{u^-} \left[D_{(u,\rho)}\left(\lambda_s \right)\right], \quad a.s.$$
    Hence, we get
    \begin{align*}
        F^T &= \frac{1}{\sqrt{T}} \int_0^T \left(\int_{\RR_+^2} \EE_{u^-} \left[D_{(u,\rho)}\left(\lambda_s \right)\right] \left( N(du,d\rho) - du d\rho \right) \right) ds \qquad a.s. \\
        &= \frac{1}{\sqrt{T}} \int_{[0,T]\times \RR_+}  \left(\int_u^T \EE_{u^-} \left[D_{(u,\rho)}\left(\lambda_s \right)\right]ds \right)\left( N(du,d\rho) - du d\rho \right) \qquad a.s.
    \end{align*}
    Now, by setting for any $(u,\rho)\in \RR^2_+$
    $$\z_{(u,\rho)} = \1_{\rho \leq \lambda_u} \quad \text{and} \quad Z^T_{(u,\rho)} = \frac{1}{\sqrt{T}} \int_u^T \EE_{u^-}\left[ D_{(u,\rho)}\left(\lambda_s \right)\right] ds,$$
    we may write $F^T$ as a divergence term:
    $$F^T = \delta \left( \z Z^T \right) \quad a.s.$$
    The rest of the proof is widely inspired by the proof of Theorem 3.2.7 of \cite{coutin_normal_2026} and we refer to this paper for more details. In particular, this theorem provides an upper bound of the $1$-Wasserstein distance:
    \begin{align*}
            \Wcal_1\left(F^T, \Ncal(0,\sigma^2)\right) &\leq \E{  \left|\sigma^2- \int_{\RR_+^2} \z_{(u,\rho)} \left|Z^T_{(u,\rho)}\right|^2 du d\rho \right|} + \E{\int_{\RR_+^2}\z_{(u,\rho)}\left|Z^T_{(u,\rho)} \right|^3 dud\rho}\\
            & \quad +\E{ \iint_{\RR_+^2} \z_{(u,\rho)} Z^T_{(u,\rho)} \left|\delta \left( \z D_{(u,\rho)}\left(Z^T \right)\right) \right| dud\rho} \\
            & \quad +\E{ \iint_{\RR_+^2} \z_{(u,\rho)} Z^T_{(u,\rho)} \left|\delta \left( \z D_{(u,\rho)}\left(Z^T \right)\right) \right|^2 dud\rho}\\
            & \quad + 2\E{ \iint_{\RR_+^2} \z_{(u,\rho)} \left|Z^T_{(u,\rho)} \right|^2 \left|\delta \left( \z D_{(u,\rho)}\left(Z^T \right)\right) \right|dud\rho}.
        \end{align*}
    Note that, in the setting of the theorem, $Z^T$ does not depend on the parameter $\rho$ and we have
    $$Z^T_{(u,\rho)} = \frac{1}{\sqrt{T}}\int_u^T f(s-u) ds .$$
    Then the Malliavin derivative of $Z^T$ is equal to $0$ which provides
    \begin{align*}
            \Wcal_1\left(F^T, \Ncal(0,\sigma^2)\right) \leq & \E{  \left|\sigma^2- \int_{\RR_+^2} \z_{(u,\rho)}\left|Z^T_{(u,\rho)}\right|^2 du d\rho \right|} + \E{\int_{\RR_+^2}\z_{(u,\rho)}\left[Z^T_{(u,\rho)} \right|^3 dud\rho}
    \end{align*}
    Note that
    \begin{align*}
        \E{\int_{\RR_+^2}\z_{(u,\rho)}\left[Z^T_{(u,\rho)} \right|^3 dud\rho}=& \frac{1}{T^{3/2}}\E{\int_0^T \lambda_u \left(\int_u^T \EE_{u^-}\left[ D_{(u,0)}(\lambda_s)\right] ds \right)^3 du}.
    \end{align*}
    Hence, the assumptions of the theorem yields
    \begin{align*}
        T_2 \leq \frac{1}{T^{3/2}}\E{\int_0^T \lambda_u \left(\int_u^T f(s-u) ds \right)^3 du} \leq \frac{\left\| f\right\|_{\LL^1}}{\sqrt{T}}\sup_{x\geq 0}\E{\lambda_x}
    \end{align*}
    Therefore, there exists $C>0$ such that for every $T>0$,
    \begin{align*}
            \Wcal_1\left(F^T, \Ncal(0,\sigma^2)\right) \leq & \E{  \left|\sigma^2- \frac{1}{T}\int_{0}^T \lambda_u \left(\int_u^T f(s-u) ds \right)^2 du  \right|} + \frac{C}{\sqrt{T}}
    \end{align*}
\end{proof}

\section{Application to Subcritical Linear Hawkes process}
We apply Theorem \ref{thm: main} to the intensity of Hawkes processes. We first recall the definition of a Hawkes process and introduce the associated intensity function. We then use the general result to derive a second theorem, providing a convergence result and an explicit upper bound on the convergence rate.

\begin{definition}[linear Hawkes process]
\label{def:Hawkes}
Let $(\Omega, \FF, \PP)$ be a general probability space.
Let $\mu\in \RR_+$ and $\phi:\RR_+ \to \RR_+$. A linear Hawkes process $H:=(H_t)_{t \geq 0}$ with baseline $\mu$ and kernel $\phi$ is a counting process such that  
\begin{itemize}
\item[(i)] $H_0=0,\quad \PP-a.s.$,
\item[(ii)] its ($\FF$-predictable) intensity process is given by
$$\lambda_t:= \mu + \int_{(0,t)} \phi(t-s) dH_s, \quad t>0,$$
that is for any $0\leq s \leq t $ and $A \in \Fcal_s$, $ \E{\1_A (H_t-H_s)} = \E{\int_{(s,t]} \1_A \lambda_r dr }$.
\end{itemize}
\end{definition}

\begin{remark}[Stationary Hawkes process]
    Another important version of an Hawkes process is the stationary one. In the following, we will write $\lambda^\infty$ the intensity associated to the stationary Hawkes process. We refer to \cite{bremaud_stability_1996} for more details on stability and stationarity of Hawkes processes.
\end{remark}

\begin{theorem}
\label{thm:appli_hawkes}
   Let $H$ be a linear Hawkes process with baseline intensity $\mu>0$ and kernel $\phi \in \LL^1$. We denote by $\lambda$ its associated intensity. Assume that
    $$\|\phi\|_{\LL^1} < 1 \quad \text{and} \quad \int_0^\infty t \phi(t) dt < +\infty.$$
    Then, there exist $C>0$ and $\sigma>0$ such that for any $T>0$,
    \begin{align*}
            \Wcal_1\left(\frac{\int_0^T \lambda_s - \E{\lambda_s} ds}{\sqrt{T}} ,\, \Ncal(0,\sigma^2)\right) \leq & C\left( \frac{1}{\sqrt{T}} + \int_T^\infty \psi(s)ds \right).
    \end{align*}
    Here, $\psi \in \LL^1$ stands for the resolvent associated with the Volterra equation with kernel $\phi$ i.e. for any $x\in \RR_+$, $\psi(x) = \phi(x) +\psi* \phi(x) = \phi(x) + \int_0^x \phi(y)\psi(x-y) dy .$
\end{theorem}

\begin{proof}
    Note that $\lambda$ respects the assumptions of Theorem \ref{thm: main}. Indeed, we can refer to \cite{coutin_normal_2026} for the proof of the different points (see Corollary A.2.6 and Lemma A.2.8). In particular, in the setting of Hawkes processes, for every $(u,s,\rho)\in \RR_+^3$ with $u<s$, it holds
    $$\1_{\rho \leq \lambda_u}\EE_{u^-} \left[ D_{(u,\rho)}(\lambda_s)\right] =\1_{\rho \leq \lambda_u} \psi(s-u),$$
    with $\psi$ the resolvent associated with the Volterra equation with kernel $\phi$. Besides, for any $x\in \RR_+$,
    \begin{equation}
        \psi(x) = \phi(x) +\psi* \phi(x) = \phi(x) + \int_0^x \phi(y)\psi(x-y) dy.
    \end{equation}
    which, under the assumption $\left\|\phi \right\|_{\LL^1}<1$, yields $\psi \in \LL^1$ and
    \begin{equation}
        \left\|\psi \right\|_{\LL^1} = \frac{\left\|\phi \right\|_{\LL^1}}{1-\left\|\phi \right\|_{\LL^1}}.
    \end{equation}
    We now can apply Theorem \ref{thm: main} and therefore for any $\sigma>0$, there exists $C>0$ such that for any $T>0$, 
    \begin{align}
    \label{eq:maj_dist_wassersteint}
            \Wcal_1\left(\frac{\int_0^T \lambda_s - \E{\lambda_s} ds}{\sqrt{T}}, \Ncal(0,\sigma^2)\right) \leq & \E{  \left|\sigma^2- \int_{0}^T \lambda_u \left(\int_u^T \psi(s-u) ds \right)^2 du  \right|} + \frac{C}{\sqrt{T}}.
    \end{align}
    For the rest of the proof, we take $\sigma^2 = \E{\lambda^\infty_0} \left\|\psi \right\|_{\LL^1}^2$ where $\lambda^\infty$ is the stationary version of $\lambda$ and we denote by $R^T$ the non-trivial term of the right hand side of the inequality:
    $$R^T := \E{\left|\E{\lambda^\infty_0} \left\|\psi \right\|_{\LL^1}^2- \frac{1}{T}\int_{0}^T \lambda_u \left(\int_u^T \psi(s-u) ds \right)^2 du  \right|}.$$
    Then, by the triangle inequality,
    \begin{align}
        R^T \leq & \frac{\left\| \psi \right\|_{\LL^1}^2}{T}\E{\left|\int_0^T \E{\lambda^\infty_0} - \lambda^\infty_u du \right|} + \frac{\left\| \psi \right\|_{\LL^1}^2}{T}\E{\left|\int_0^T \lambda^\infty_u - \lambda_u du \right|} \nonumber\\
        &+ \frac{1}{T} \E{\left|\int_0^T \lambda_u \left(\left\|\psi \right\|_{\LL^1}^2 -\left|\int_u^T \psi(s-u) ds\right|^2 \right) du\right|}
        \label{eq: R_T}
    \end{align}
    Note that the control of the two first terms of the right hand side of the inequality has already been done in \cite{coutin_normal_2026}. More precisely, these terms correspond respectively to $T_{1,1}$ and $T_{1,2}$ in the proof of Theorem 4.3.18 \cite{coutin_normal_2026} and they are bounded as follows:
    \begin{align}
        \frac{\left\| \psi \right\|_{\LL^1}^2}{T}\E{\left|\int_0^T \E{\lambda^\infty_0} - \lambda^\infty_u du \right|} &\leq \frac{\left\| \psi \right\|_{\LL^1}^3 \sqrt{2\E{\lambda^\infty_0}}}{\sqrt{T}},\label{eq:maj_L1_fluctuation_int_stat}\\
        \text{and} \quad \frac{\left\| \psi \right\|_{\LL^1}^2}{T}\E{\left|\int_0^T \lambda^\infty_u - \lambda_u du \right|} &\leq \left(\frac{\left\| \psi \right\|_{\LL^1}^2\E{\lambda^\infty_0}}{1- \left\|\phi \right\|_{\LL^1}} \int_0^\infty t\phi(t) dt \right)\frac{1}{T} .
        \label{eq:maj_L1_int_stat}
    \end{align}
Regarding the last right hand term of \eqref{eq: R_T}, we can use the fact that $a^2-b^2 \leq 2 \max(a,b)\left|a-b\right|$ to get
\begin{align*}
    \frac{1}{T} \E{\left|\int_0^T \lambda_u \left(\left\|\psi \right\|_{\LL^1}^2 -\left|\int_u^T \psi(s-u) ds\right|^2 \right) du\right|} \leq& \frac{2 \left\| \psi \right\|_{\LL^1}}{T} \E{\int_0^T \lambda_u \left|\left\|\psi \right\|_{\LL^1} -\int_u^T \psi(s-u) ds \right| du}\\
    \leq& \frac{2 \left\| \psi \right\|_{\LL^1}}{T} \sup_{x\in \RR_+}\E{\lambda_x}\int_0^T \left|\left\|\psi \right\|_{\LL^1} -\int_u^T \psi(s-u) ds \right| du \\
    \leq& \frac{2 \left\| \psi \right\|_{\LL^1}}{T} \sup_{x\in \RR_+}\E{\lambda_x} \int_0^T \int_{T-u}^{+\infty} \psi(s) ds du .
\end{align*}
Yet, Fubini's theorem yields
\begin{align*}
    \int_0^T \int_{T-u}^{+\infty} \psi(s) dsdu &= \int_0^{+\infty} \int_{\max\{T-s,0\}}^T \psi(s) duds\\
    &= \int_0^{+\infty} \left(T- \max\{T-s,0\} \right)\psi(s) ds\\
    &= \int_0^{T} s\psi(s) ds + T \int_T^{+\infty} \psi(s) ds
\end{align*}
Hence, we obtain the following upperbound:
\begin{align}
    \frac{1}{T} \E{\left|\int_0^T \lambda_u \left(\left\|\psi \right\|_{\LL^1}^2 -\left|\int_u^T \psi(s-u) ds\right|^2 \right) du\right|} \leq & \frac{2 \left\| \psi \right\|_{\LL^1}}{T} \sup_{x\in \RR_+}\E{\lambda_x} \int_0^T s\psi(s)  ds\nonumber \\
    &+ 2 \left\| \psi \right\|_{\LL^1} \sup_{x\in \RR_+}\E{\lambda_x} \int_T^\infty \psi(s)ds .\label{eq:maj_diff_psi_int}
\end{align}
We end the proof by proving that under the assumptions of the theorem, $s\mapsto s\psi(s)$ belongs to $\LL^1$. Let $A>0$. Then, since $\psi = \phi + \phi*\psi$, it holds that
\begin{align*}
    \int_0^A t \psi(t) dt = \int_0^A t \phi(t) dt + \int_0^A t\int_0^t \phi(s)\psi(t-s)ds dt.
\end{align*}
Using Fubini's theorem combined with $t=(t-s)+s$, we get
\begin{align*}
    \int_0^A t \psi(t) dt =& \int_0^A t \phi(t) dt + \int_0^A \int_s^A  (t-s)\phi(s)\psi(t-s)dt ds +\int_0^A \int_s^A  s\phi(s)\psi(t-s)dt ds\\
    =&  \int_0^A t \phi(t) dt + \int_0^A \phi(s) \int_s^A  (t-s)\psi(t-s)dt ds +\int_0^A s\phi(s) \int_s^A  \psi(t-s)dt ds.
\end{align*}
Since $\psi\in \LL^1$ and $\int_0^{+\infty} s\phi(s) ds <+\infty$, we obtain
\begin{align*}
    \int_0^A t \psi(t) dt =& \int_0^A t \phi(t) dt + \int_0^A \int_s^A  (t-s)\phi(s)\psi(t-s)dt ds +\int_0^A \int_s^A  s\phi(s)\psi(t-s)dt ds\\
    \leq &  \int_0^{+\infty} t \phi(t) dt + \int_0^A \phi(s) \int_s^A  (t-s)\psi(t-s)dt ds +\left\|\psi \right\|_{\LL^1}\int_0^{+\infty} s\phi(s) ds\\
    \leq & \left( 1+\left\|\psi \right\|_{\LL^1}\right) \int_0^{+\infty} s\phi(s) ds + \int_0^A \phi(s) \int_s^A  (t-s)\psi(t-s)dt ds
\end{align*}
The change of variable $t=t-s$ yields
\begin{align*}
    \int_0^A t \psi(t) dt \leq & \left( 1+\left\|\psi \right\|_{\LL^1}\right) \int_0^{+\infty} s\phi(s) ds + \int_0^A \phi(s) \int_0^{A-s}  t\psi(t)dt ds\\
    \leq & \left( 1+\left\|\psi \right\|_{\LL^1}\right) \int_0^{+\infty} s\phi(s) ds + \left\| \phi \right\|_{\LL^1}  \int_0^{A}  t\psi(t)dt .
\end{align*}
Here, we have also made use of $\int_0^{A-s}  t\psi(t)dt \leq \int_0^{A}  t\psi(t)dt$. Therefore, we get
\begin{align*}
    \int_0^A t \psi(t) dt \leq & \frac{ 1+\left\|\psi \right\|_{\LL^1}}{1-\left\| \phi \right\|_{\LL^1}} \int_0^{+\infty} s\phi(s) ds = \frac{1}{(1-\left\| \phi \right\|_{\LL^1})^2}\int_0^{+\infty} s\phi(s) ds.
\end{align*}
Since the function $A\mapsto \int_0^{A}  t\psi(t)dt$ is bounded and non-decreasing, it converges as $A \to +\infty$ and so 
\begin{equation}
\label{eq:maj_moment_1_psi}
    \int_0^{+\infty} t \psi(t) dt \leq \frac{1}{(1-\left\| \phi \right\|_{\LL^1})^2}\int_0^{+\infty} s\phi(s) ds.
\end{equation}
Hence, plugging \eqref{eq:maj_moment_1_psi} in \eqref{eq:maj_diff_psi_int}, we get:
\begin{align}
    \frac{1}{T} \E{\left|\int_0^T \lambda_u \left(\left\|\psi \right\|_{\LL^1}^2 -\left|\int_u^T \psi(s-u) ds\right|^2 \right) du\right|} \leq & \frac{2 \left\| \psi \right\|_{\LL^1}}{T}  \frac{\sup_{x\in \RR_+}\E{\lambda_x}}{(1-\left\| \phi \right\|_{\LL^1})^2}\int_0^{+\infty} s\phi(s) ds\nonumber \\
    &+ 2 \left\| \psi \right\|_{\LL^1} \sup_{x\in \RR_+}\E{\lambda_x} \int_T^\infty \psi(s)ds.
    \label{eq:maj_diff_psi}
\end{align}
Plugging \eqref{eq:maj_L1_fluctuation_int_stat}, \eqref{eq:maj_L1_int_stat} and \eqref{eq:maj_diff_psi} into \eqref{eq: R_T}, we get a constant $C>0$ such that for any $T>0$,
\begin{align*}
    R^T \leq C\left( \frac{1}{\sqrt{T}}+ \int_T^\infty \psi(t)dt\right).
\end{align*}
This with \eqref{eq:maj_dist_wassersteint} conclude the proof.
\end{proof}

\paragraph{Acknowledgment}: This work was supported by the French National Research Agency (ANR) through the HAPPY project (ANR-23-CE40-0007).

\end{document}